\documentclass[12pt,reqno]{amsart}
\usepackage{graphicx}
\usepackage{amssymb}

\numberwithin{equation}{section}
\newtheorem{Theorem}{Theorem}[section]
\newtheorem{Lemma}[Theorem]{Lemma}
\newtheorem{Corollary}[Theorem]{Corollary}

\newtheorem{Remark}[Theorem]{Remark}

\newtheorem{Example}[Theorem]{Example}

\newtheorem{Definition}[Theorem]{Definition}

\numberwithin{equation}{section}

\begin{document}

\title[ON APPROXIMATELY COHEN-MACAULAY BINOMIAL EDGE IDEAL]{ON APPROXIMATELY COHEN-MACAULAY BINOMIAL EDGE IDEAL}%
\author{Sohail Zafar}%
\address{Abdus Salam School of Mathematical Sciences, GCU, Lahore Pakistan}%
\email{sohailahmad04@gmail.com}%

\thanks{This research was partially supported by Higher Education Commission, Pakistan}

\begin{abstract}
Binomial edge ideals $I_{G}$ of a graph $G$ were introduced by \cite{herzog}. They found some classes of graphs $G$ with the property that $I_{G}$ is a Cohen-Macaulay ideal. This might happen only for few classes of graphs. A certain generalization of being Cohen-Macaulay, named approximately Cohen-Macaulay, was introduced by S. Goto in \cite{goto}. We study classes of graphs whose binomial edge ideal are approximately Cohen-Macaulay. Moreover we use some homological methods in order to compute their Hilbert series.
\end{abstract}
\maketitle
\begin{center}

INTRODUCTION
\end{center}

Let $K$ denote a field. Let $G$ denote a connected undirected graph over the
vertices labeled by $[n]=\{1,2,\dots,n\}.$ Let $R=K[x_{1},\dots,x_{n}]$ denote
the polynomial ring in $n$ variables over the field $K$. The edge ideal $I$ of $G$ is generated
by all $x_{i}x_{j}$ , $i<j$ , such that $\{i,j\}$ forms an edge of $G.$ This
notion was studied by Villarreal \cite{villarreal} where it is also discussed under
which circumstances $R/I$ is a Cohen-Macaulay ring. It seems to be hopeless
to characterize all of the graphs $G$ such that $R/I$ is a Cohen-Macaulay
ring.

In a similar way one might define the binomial edge ideal $I_{G}\subseteq
S=K[x_{1},\dots,x_{n},y_{1},\dots,y_{n}].$ It is generated by all binomials $%
x_{i}y_{j}-x_{j}y_{i}$ , $i<j$ , such that $\{i,j\}$ is an edge of $G.$ In
the paper of V. Ene, J. Herzog, T. Hibi \cite{ene}, the authors start with the
systematic investigation of $I_{G}$. There is a primary decomposition of $%
I_{G}$. Moreover there are examples such that $S/I_{G}$ is a Cohen-Macaulay
ring. Furthermore, the authors believe that it is hopeless to characterize
those $G$ such that the binomial edge ideal is Cohen-Macaulay in general.

A generalization of the notion of a Cohen-Macaulay ring was introduced by S.
Goto \cite{goto} under the name approximately Cohen-Macaulay (see also the
Definition 1.4). Not so many examples of approximately Cohen-Macaulay rings are known. Note
that they are not a domain. In the present paper we collect a few graphs $G$
such that the associated ring $S/I_{G}$ is an approximately Cohen-Macaulay
ring, where $I_{G}$ denotes the binomial edge ideal.

In fact, we give a complete characterization of all connected trees whose binomial edge ideal are approximately Cohen-Macaulay. Such trees are described as $3$-star like. We also prove that the cycle of length $n$ is approximately Cohen-Macaulay.

It is well known fact that the canonical module $\omega (S/I)$ of a Cohen-Macaulay ring is a Cohen-Macaulay module. The converse is not true.
Examples of non Cohen-Macaulay rings such that $\omega (S/I)$ is Cohen-Macaulay are approximately Cohen-Macaulay rings (as follows by the definition). So we
investigate the canonical modules and the modules of deficiency. As
applications of our investigations we study the Hilbert series of our
examples.

The paper is organized as follows: In the first section there is a summary of preliminary results. In section 2 we give the characterization of approximately Cohen-Macaulay binomial edge ideals for trees. In section 3 we prove that the binomial edge ideal of any cycle is approximately Cohen-Macaulay. In addition we compute the Hilbert series of the corresponding ideals. In order to do that for the $n$-cycle, we include some investigations on the canonical module of the binomial edge ideal of a complete graph.
\section{PRELIMINARY RESULTS}

In this section we will fix the notations we use in the sequel. Moreover we
summarize auxiliary results that we need in our paper. We have tools from
combinatorial algebra, canonical modules and commutative algebra. In
general we fix an arbitrary field $K$ and the polynomial ring $%
S=K[x_{1},\dots,x_{n},y_{1},\dots,y_{n}]$ in $2n$ variables $x_{1},\dots,x_{n}$
and $y_{1},\dots,y_{n}.$

\bigskip
\bigskip
\noindent {\Large (a) Combinatorial Algebra.}

By $G$ a graph on the vertex set $[n]$ we always understand a connected simple
graph. The binomial edge ideal $J_{G}\subset S$ is defined as the ideal
generated by all binomials $f_{ij}=x_{i}y_{j}-x_{j}y_{i},$ $1\leq i<j\leq n,$
where $\{i,j\}$ is an edge of G. This construction has been found in \cite{herzog}. For a subset  $T\subset \lbrack n]$ let be $G_{[n]\backslash{T}}$ the graph
obtained from $G$ by deleting all vertices that belongs to $T$. Let $\ c=c(T)
$ denote the number of connected components of $G_{[n]\backslash{T}}$. Let  $%
G_{1},...G_{c}$ \ denote these components of $G_{[n]\backslash{T}}.$ Let $\tilde{G}$
denotes the complete graph on the vertex set of $G$. Define
$$P_{T}(G)=(\cup _{i\in T}\{x_{i},y_{i}\},J_{\tilde{G}_{1}},\dots,J_{\tilde{G}%
_{C(T)}})$$
where $\tilde{G}_{i},$ $i=1,...,c=c(T),$ denotes the complete graph on the
vertex set of the connected component of $G_{i}.$ Then $P_{T}(G)\subset S$
is a prime ideal of height $n+|T|-c,$ where $|T|$ denotes the number of
elements of $T$. \ Moreover
$$J_{G}=\cap _{T\subseteq \lbrack n]}P_{T}(G).$$
For these and related facts we refer to \cite{herzog}. For us the following lemma is
important.

\begin{Lemma}
With the previous notations it follows that  $J_{G}\subset P_{T}(G)$ is a
minimal prime if and only if either  $T=\emptyset $ \ or $T\neq \emptyset $
\ and for each $i\in T$ \ we have that  $c(T\backslash\{i\})<c(T).$
\end{Lemma}
For the proof see corollary 3.9 of \cite{herzog}. It is noteworthy to say that  $J_{G}$
is the intersection of prime ideals. That is, $S/J_{G}$ is reduced. Finally
we observe that $S/J_{G}$ is a graded ring with natural grading induced
by the grading of $S$.

\bigskip
\noindent {\Large (b) Canonical Modules.}

Let $S$ as above denote the polynomial ring. Let $M$ denote the finitely
generated graded $S-$module. For some of our arguments we use a few basic
facts about duality. In the following we introduced the modules of
deficiency related to $M$.

\begin{Definition}

For an integer $i$ define
$$\omega ^{i}(M):=Ext_{S}^{2n-i}(M,S)$$
the $i-th$ module deficiency  to $M$. Note that $\omega ^{i}(M)=0$ for $%
i<depth(M)$ and $i>dim(M)=d.$ For $i=d$ we call
$$\omega (M)=\omega ^{d}(M)$$
the canonical module of $M$.
\end{Definition}
These modules have been introduced in \cite{Schenzel}. See also \cite{Schenzel} for some properties
of $\omega ^{i}(M)$ and related facts. In particular we have
$$dim(\omega (M))=\dim (M)$$

and $dim(\omega ^{i}(M))\leq i$ for all $0\leq i<d.$

\begin{Lemma}

With the previous notation let $I\subset S$ be the homogenous ideal then :

\begin{enumerate}
\item There is a natural homomorphism
$$
S/I\rightarrow \omega (\omega (S/I))\cong Hom(\omega (S/I),\omega (S/I))
$$
which is an isomorphism if $S/I$ is a Cohen-Macaulay ring.

\item Suppose that $S/I$ is Cohen-Macaulay ring with $dim(S/I)>0.$ Suppose
that $\omega (S/I)$ is an ideal of $S/I$ then $(S/I)/\omega (S/I)$ is a
Gorenstein ring with $dim(S/I)-1.$
\end{enumerate}
\end{Lemma}

For the proof we refer to \cite{Schenzel} and \cite{burns}.

\bigskip
\noindent {\Large (c) Approximately Cohen-Macaulay rings.}

For our purpose here we need a certain generalization of Cohen-Macaulay rings
that was originally introduced by S.Goto (see \cite{goto})
in the case of local rings.

\begin{Definition}

Let $R$ denote a commutative ring of finite dimension $d$ and $I\subset R$
an ideal. Then

\begin{enumerate}
\item $\mathbf{\ }Assh_{R}(R/I)=\{p\in Ass_{R}(R/I):\dim R/p=\dim R/I\}.$

\item $U_{R}(I)=\cap _{p\in Assh_{R}(R/I)}I(p)$ where $I=\cap _{p\in
Ass_{R}(R/I)}I(p)$ denotes a minimal primary decomposition of the ideal $I$.%
\textbf{\ }
\end{enumerate}

That is $U_{R}(I)$ describes the equidimensional part of the primary
decomposition of the ideal $I.$
\end{Definition}
As an analogue to Goto's notion of approximately Cohen-Macaulay rings we
define a graded version of it.

\begin{Definition}
Let $R=\underset{i\geq 0}{\oplus }R_{i}$ denote a standard $K-$algebra,
where $K=R_{0}.$ It is called approximately Cohen-Macaulay if $R_{+}=%
\underset{i>0}{\oplus }R_{i}$ contains a homogeneous element $x$ such that $%
R/x^{n}R$ is a Cohen-Macaulay ring of $dim(R)-1$ for all $n\geq 1.$
\end{Definition}
A characterization of approximately Cohen-Macaulay rings can be done by the
following theorem.

\begin{Theorem}

Let $I\subset S$ denote a homogenous ideal and $d=dim(S/I).$ Then
the following conditions are equivalent.

\begin{description}
\item[(i)] $S/I$ is approximately Cohen-Macaulay.

\item[(ii)] $S/U_{S}(I)$ is a $d-$dimensional Cohen-Macaulay ring and $%
depth(S/I)\geq d-1.$

\item[(iii)] $\omega ^{d}(S/I)$ is Cohen-Macaulay module of dimension $d$
and $\omega ^{d-1}(S/I)$ is either zero or a $(d-1)-$dimensional
Cohen-Macaulay module.
\end{description}
\end{Theorem}
\begin{proof}

In the case of local ring the above theorem was proved by Goto in the paper
of \cite{goto}. The proof in the graded case follows by the same
arguments. We have to adopt the graded situation from the local one. We omit
the details.
\end{proof}

\begin{Remark}

A necessary condition for $S/I$ to be an approximately Cohen-Macaulay ring
is that $\dim S/p\geq \dim S/I-1$ for all $p\in Ass(S/I)$
\end{Remark}
\section{TREES}
In this section we will characterize all trees which the property
that the associated binomial edge ideal defines an approximately
Cohen-Macaulay ring. A graph G is called tree if it is connected and has no
cycles. The simplest tree is the line. The binomial edge ideal of a line is a complete
intersection. So it defines a Cohen-Macaulay ring (see \cite{herzog}). In our consideration we do
not consider the line in detail. To classify all the trees that are approximately Cohen-Macaulay, we have to introduce a new terminology called
$3$-star like trees.
\begin{Definition}
A tree $G$ is called 3-star like if there are no vertices of degree $\geq 4$ and if either there is at most
one vertex of degree 3 or there is at most one edge with the property  that both of its
vertices are of degree 3. Then the line is of course 3-star like. Other types of 3-star like trees (see the figures 1 and 2).
\end{Definition}

\begin{figure}[h]
\includegraphics[width=7cm]{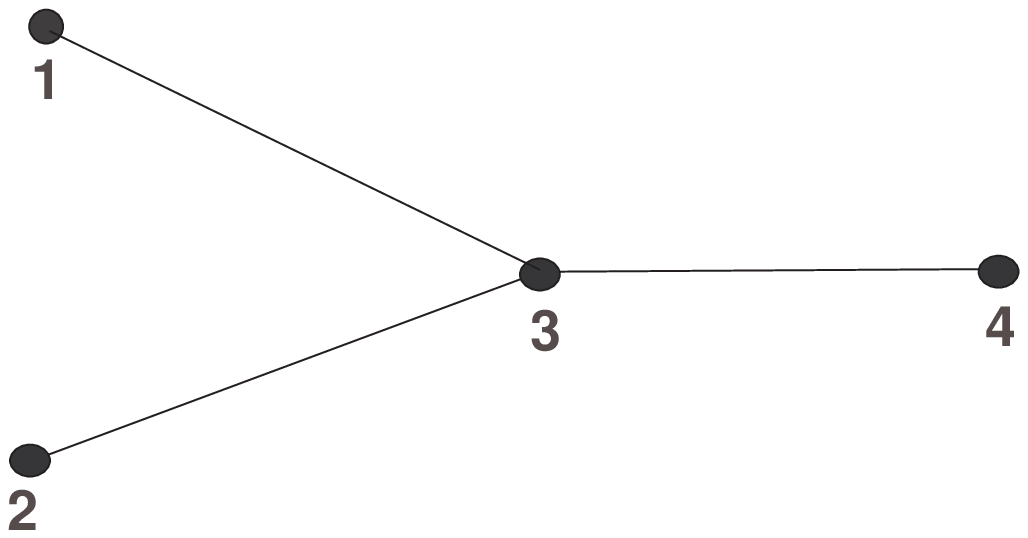}\\
\caption{}\label{}
\end{figure}

\begin{Example}

Consider the simplest example of the 1st type of 3-star like tree as shown in
figure 1. Its binomial edge ideal is approximately Cohen-Macaulay of
dimension 6, depth 5 and the Hilbert Series is%
\begin{equation*}
H(S/J_{G},t)=\frac{1+2t-2t^{3}}{(1-t)^{6}}.
\end{equation*}

\end{Example}
\begin{figure}[h]
 \includegraphics[width=7cm]{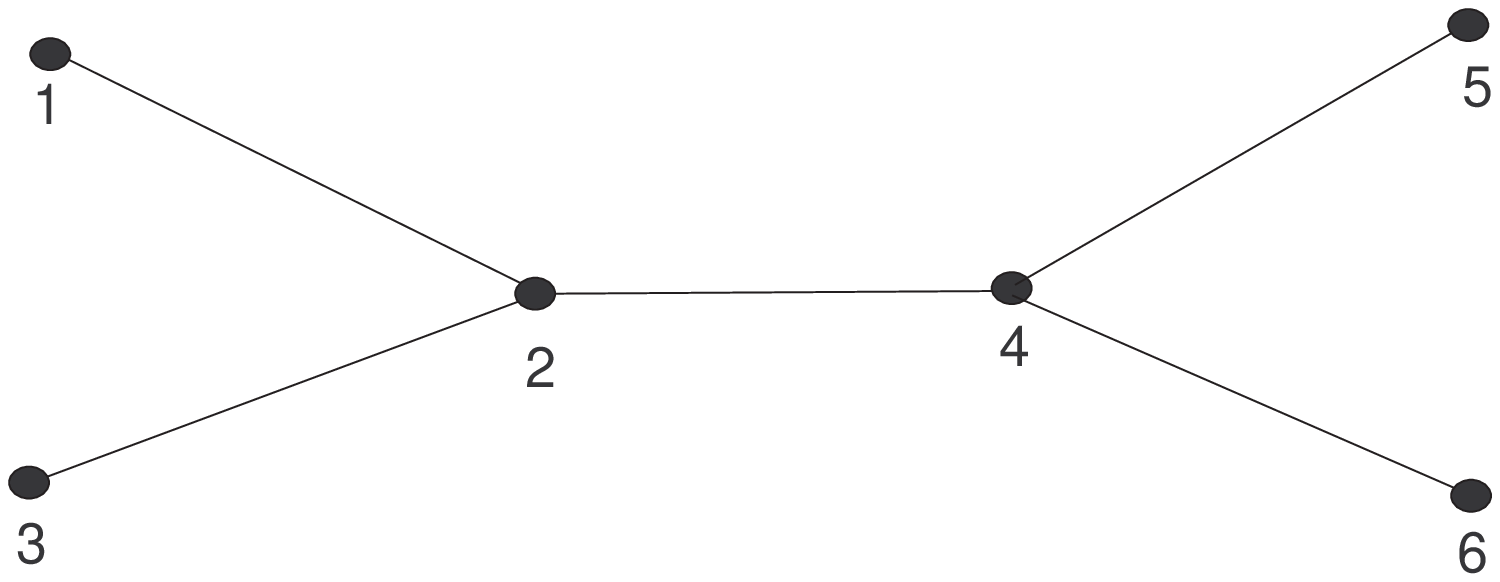}\\
 \caption{}\label{}
\end{figure}

\begin{Example}

The simplest example in case of 2nd type of 3-star like tree is shown above.
Its binomial edge ideal is approximately Cohen-Macaulay of dimension 8,
depth 7 and the Hilbert Series is%
\begin{equation*}
H(S/J_{G},t)=\frac{1+4t+5t^{2}-3t^{4}}{(1-t)^{8}}.
\end{equation*}

\end{Example}
In order to prove the approximately Cohen-Macaulay property of binomial edge ideals we need the following construction principle. It will be useful also in different circumstances.

\begin{Lemma}Let $G$ be any connected graph with vertices set $[n]$ having at least one vertex of degree 1, choose one of them and label it by $n$. Let $G^\prime$ be a graph on vertices set $[n+1]$ by attaching one edge $\{n,n+1\}$  to the graph $G$. Now  $G$ is approximately Cohen-Macaulay if and only if $G^\prime$ is approximately Cohen-Macaulay.
\end{Lemma}
\begin{proof}
Suppose that $J_{G}$ and $J_{G^\prime}$ denotes the binomial edge ideal of the corresponding graphs. Let $dim(S/J_{G})=d$ and $depth(S/J_{G})=d-1$. Now $J_{G^\prime}=
(J_{G},f)$ where $f=x_ny_{n+1}-x_{n+1}y_n$ and $S^\prime=S[x_{n+1},y_{n+1}]$, therefore $dim(S^\prime/J_{G})=d+2$ and $depth(S^\prime/J_{G})=d+1.$

Now $n\notin T$  for all  $T$ $\subseteq $ $[n]$ such that $%
c(T\backslash \{i\})<c(T).$ Which implies $x_{n},y_{n}$ $\notin P_{T}(G)$
for all $\ P_{T}(G)\in Ass(S/J_{G}),$ hence $f$ is not a zero divisor in $S/J_{G}
$ and is regular. Therefore $dim(S^\prime/J_{G^\prime})=d+1$ and $%
depth(S^\prime/J_{G^\prime})=d.$

Consider the exact sequence

$$0\rightarrow S^\prime /J_{G}(-2)\mathop\rightarrow\limits^f S^\prime /J_{G}\rightarrow S^\prime /J_{G^\prime}\rightarrow 0.$$

Apply $Hom(S^\prime, .)$ to above sequence we get the long exact sequence
\begin{multline*}
0\rightarrow \omega ^{d+2}(S^\prime /J_{G})\mathop\rightarrow\limits^f \omega^{d+2}(S^\prime /J_{G})(2)\rightarrow \omega^{d+1}(S^\prime /J_{G^\prime})\\
\rightarrow \omega^{d+1}(S^\prime /J_{G})\mathop
\rightarrow\limits^f \omega^{d+1}(S^\prime /J_{G})(2)\rightarrow \omega^{d} (S^\prime /J_{G^\prime})\rightarrow 0.
\end{multline*}

Clearly $f$ is $\omega^{d+2}(S^\prime /J_{G})-regular$. Moreover $f$ is $\omega^{d+1}(S^\prime /J_{G})-regular$ since $f$ is $S/J_G$-regular.
So multiplication by $f$ is injective and we get the following two isomorphisms:

$$\omega^{d+1}(S^\prime /J_{G^\prime})\cong \omega^{d+2}(S^\prime /J_{G})/f\omega^{d+2}(S^\prime /J_{G})(2)$$
and
$$\omega^{d}(S^\prime /J_{G^\prime})\cong \omega^{d+1}(S^\prime /J_{G}))/f\omega^{d+1}(S^\prime /J_{G})(2).$$

Hence $\omega
^{d+2}(S^\prime/J_{G})$ and $\omega ^{d+1}(S^\prime/J_{G})$ is of dimension $d+2$ and $d+1$ resp. Therefore both are Cohen-Macaulay if and only if $\omega^{d+1}(S^\prime /J_{G^\prime})$ and $\omega^{d}(S^\prime /J_{G^\prime})$ are Cohen-Macaulay of dimension $d+1$ and $d$ resp.
Therefore from theorem 1.6 $S^\prime/J_{G^\prime}$ is approximately Cohen-Macaulay if and only if $S /J_{G}$ is approximately Cohen-Macaulay.

\end{proof}

\begin{Corollary}Let $G$ be a tree with vertices set $[n]$,
then $S/J_{G}$ is approximately Cohen-Macaulay if and only if G is
3-star like.
\end{Corollary}
\begin{proof}
Suppose the contrary that  $G$ is not 3-star like then we have two cases:

\begin{description}
\item[Case 1] If $G$ contains at least one vertex of degree $d\geq 4$ say $i.
$ Then $T=\{i\}$ and $c(T)=d$ so we have $\dim(S/P_{T}(G))=n+d-1\geq n+3.$

\item[Case 2] If G has two vertices of degree $3$ which are not adjacent say
$i$ and $j.$ Then $T=\{i,j\}$ and  $c(T)=5.$ Therefore we have $%
dim(S/P_{T}(G))=n+3.$
\end{description}
Therefore in both cases $dim(S/J_{G})\geq n+3.$

On the other hand  $J_{\tilde{G}}$ $\in Ass(S/J_{G})$ and $%
dim(S/J_{\tilde{G}})=n+1$ so $depth(S/J_{G})\leq n+1$ and hence $S/J_{G}$
is not approximately Cohen-Macaulay.

Conversely, in order to prove that any $3$-star like tree is approximately
Cohen-Macaulay we will use induction on $n,$ the number of vertices. For the
case of a line there is nothing to prove. For $n=4$
and $n=6$ it is true see example 2.2. and 2.3 above.
The general case follows by the construction principle in Lemma 2.4.
\end{proof}

Now we will discuss some properties of the trees which are approximately
Cohen-Macaulay. In Corollary 2.5. we have shown that  3-star like  trees
with vertices set $[n]$ have dimension $n+2$ and depth $n+1$. Now the Hilbert
series of of 3-star like trees can be easily computed.

\begin{Lemma}
    With the notations of Lemma 2.4, We have $$H(S^\prime /J_{G^\prime},t)=(1-t^{2})H(S^\prime /J_{G},t).$$
\end{Lemma}
\begin{proof}
Consider the exact sequence $$0\rightarrow S^\prime /J_{G}(-2)\mathop\rightarrow\limits^f S^\prime /J_{G}\rightarrow S^\prime /J_{G^\prime}\rightarrow 0.$$ Hence we have a required result.
\end{proof}

\begin{Corollary}Let $G$ be a 3-star like tree
with vertex set $[n].$ Then the Hilbert Series of $S/J_{G}$ is
\begin{eqnarray*}
H(S/J_{G},t) &=&\frac{(1+2t-2t^{3})(1+t)^{n-4}}{(1-t)^{n+2}}\text{ \ for }n>3
\\
&&\text{and \ }\frac{(1+4t+5t^{2}-3t^{4})(1+t)^{n-6}}{(1-t)^{n+2}}\text{ for
}n>5\text{ respectively}
\end{eqnarray*}
for the first resp. the second type of $3$-starlike trees.
\end{Corollary}
\begin{proof}
We will prove this by induction on $n$. Consider the first case of $3$-starlike
trees. For $n=4$ it is true, see example
2.2 resp. 2.3. Suppose the claim is true for $n.$ That is,
 $$H(S^\prime /J_{G},t)=\frac{(1+2t-2t^{3})(1+t)^{n-4}}{(1-t)^{n+4}}.$$ Now by Lemma 2.6. we have
\begin{equation*}
H(S^\prime /J_{G^\prime},t)=\frac{(1+2t-2t^{3})(1+t)^{n-3}}{(1-t)^{n+3}}
\end{equation*}%
as required.

Similar arguments might be used in order to calculate the Hilbert Series
in the second case of $3$-starlike trees.
\end{proof}

\section{CYCLE}
In this section we study the algebraic properties of the  binomial edge ideal associated to a cycle. A graph G is called cycle if it is a closed directed path, with no repeated vertices other than the starting and ending vertices. We denote the cycle of length $n$ by $C$ and its binomial edge ideal by $I_{C}$.
To study further properties of $I_{C}$ or $S/I_{C}$ we shall need few basic properties from \cite{herzog}.
If $$I_{C}=\cap _{T\subseteq \lbrack n]}P_{T}(C)$$
then
$$
dimS/P_{T}(C)=n+1\text{\ \ }if\text{ \ }T=\emptyset \text{ \ and \ }\
dimS/P_{T}(C)\leq n\text{ \ }if\ \ T\not=\emptyset.$$

 Hence $$dim(S/I_{C})=n+1\text{ \ }and\text{ \ }U_{S}(I_{C})=P_{\emptyset}(C)=J_{\tilde{G}}.$$
Moreover $P_{T}(C)$ is minimal prime of $S/I_{C}$ if either (i) $T=\emptyset$
or (ii) if $T\neq \emptyset $ \ and $\ |T|>1$ and no
two elements $i,j\in T$ \ belongs to the same edge of C.

\begin{Theorem}$%
x_{1},y_{1}-x_{2},\dots,y_{n-1}-x_{n},y_{n}$ is the system of parameters for $S/I_{C}$.
\end{Theorem}
\begin{proof}
Let $\underline{x}=x_{1},y_{1}-x_{2},\dots,y_{n-1}-x_{n},y_{n}$ , $M=S/I_{C}$
\ and $I=Ann(M/\underline{x}M).$ Now $M/\underline{x}M=S/(I_{C},\underline{x}%
) $ . If we replace $y_{1}=x_{2},y_{2}=x_{3},\dots,y_{n-1}=x_{n}$ in $%
I_{C}, $ we get

$$I=(x_{1},x_{1}x_{3}-x_{2}^{2},x_{2}x_{4}-x_{3}^{2},\dots,x_{n-2}x_{n}-x_{n-1}^{2},x_{n}^{2},x_{n}x_{2}).$$
\noindent Clearly $x_{1},x_{n}\in Rad(I),$ we need to prove that $%
x_{2},\dots,x_{n-1}\in Rad(I).$ If $x_{1},x_{1}x_{3}-x_{2}^{2}\in I$ then $%
x_{2}^{2}\in I$. It follows that $x_{2}\in Rad(I)$, hence we have a basis of
induction. If $x_{k}\in Rad(I)$ for $2\leq k\leq n-2$ then $%
x_{k}x_{k+2}-x_{k+1}^{2}\in I$ as $k+1\leq n-1$. Therefore $x_{k+1}^{2}\in
Rad(I)$ it follows that $x_{k+1}\in Rad(I).$ This then implies that $%
Rad(I)=(x_{1},x_{2},\dots,x_{n})$. Hence $I$ is $m-primary$ in the ring $%
K[x_{1},\dots,x_{n}]$ so $\underline{x}$ is the system of parameters of $%
S/I_{C}.$
\end{proof}

\begin{Lemma}Let $I_{L}$ be the binomial edge ideal of a
line $L$ of length $n$ , $g=x_{1}y_{n}-x_{n}y_{1}$ and $J_{\tilde{G}}$
be binomial edge ideal of a complete graph on $[n]$ then

\begin{description}
\item[(a)] If $I_{L}=\cap _{T\subseteq \lbrack n]}P_{T}(L)$ then $%
g\notin P_{T}(L)$ for $T\not=\emptyset.$

\item[(b)] $I_{L}=(I_{L}:g)\cap J_{\tilde{G}}.$

\item[(c)] $I_{L}:(I_{L}:g)=J_{\tilde{G}}.$

\item[(d)] $I_{L}:g=I_{L}:J_{\tilde{G}}.$
\end{description}
\end{Lemma}
\begin{proof}\begin{description}
 \item[(a)] Because of $I_{L}=\cap _{T\subseteq \lbrack n]}P_{T}(L),$ it is known from \cite%
{herzog} \ that $P_{T}(L)$ is minimal prime of $I_{L}$ if either (i) $T=\emptyset $
\ or (ii) if $T\neq \emptyset $ \ and $1,n\notin T$ and if $\ |T|>1$ then there are no
two elements $i,j\in T$\ such that $\{i,j\}$ is an edge of $L$. If $T=\emptyset $, then $P_{T}(L)=J_{\tilde{G}}$ is the ideal of complete graph. Now let $T\neq \emptyset $ and $1,n\notin T$. Suppose that $\ |T|>1$ then $g\notin P_{T}(L)$ because $x_{1},y_{1},x_{n},y_{n}$ does not belongs to $\cup _{i\in T}\{x_{i},y_{i}\}$ and $g$ does not belongs to any $J_{\tilde{G}}$ for any connected component of $[n]\backslash{T}$
\item[(b)] $I_{L}:g=\cap _{g\notin P_{T}(L)}P_{T}(L),$ Using $(a)$ we
have
$$I_{L}:g=\cap _{T\neq \emptyset }P_{T}(L)$$

and $I_{L}=\cap _{T\neq \emptyset }P_{T}(L)\cap P_{\emptyset }(L)$ therefore,%
$$I_{L}=(I_{L}:g)\cap J_{\tilde{G}}.$$
\item[(c)] Using $(b)\ $we have $$I_{L}:(I_{L}:g)=((I_{L}:g)\cap J_{\tilde{G}%
}):(I_{L}:g)=J_{\tilde{G}}:(I_{L}:g).$$
In order to finish we have to prove that $J_{\tilde{G}}=J_{%
\tilde{G}}:(I_{L}:g).$
From the definition $J_{\tilde{G}}\subseteq J_{%
\tilde{G}}:(I_{L}:g).$ Now for the other inclusion let $%
I_{L}:g=(h_{1},h_{2},...,h_{r}),$ then $J_{\tilde{G}}:(I_{L}:g)=\cap
_{i=1}^{r}J_{\tilde{G}}:h_{i}.$ Now $J_{\tilde{G}}:h_{i}=J_{\tilde{G}}$ for
at least one $i,$ since $J_{\tilde{G}}$ is a prime ideal and $h_{i}\notin $ $J_{%
\tilde{G}}$ for at least one $i$ so $J_{\tilde{G}}:(I_{L}:g)\subseteq J_{%
\tilde{G}}$ and we are done.

\item[(d)] $I_{L}:g\subseteq I_{L}:J_{\tilde{G}}$ is trivial. For
another inclusion let $f\in I_{L}:J_{\tilde{G}}$ then $fJ_{\tilde{G}%
}\subseteq I_{L}\subseteq P_{T}(L),$ now $J_{\tilde{G}}\nsubseteq P_{T}(L)$
for $T\neq \emptyset ,$ so $f\in P_{T}(L)$ for $T\neq \emptyset $ which
implies $f\in I_{L}:g.$
\end{description}
\end{proof}

\begin{Definition}\cite{peskine}
Two ideals $I$ and $J$ of height $g$ in $S$ are said to be linked
if there is a regular sequence $\alpha $ of height $g$ in their intersection such that $%
I=\alpha :J$ and $J=\alpha :I.$
\end{Definition}
It is also known from \cite{peskine} that $I$ and $J$ are two linked ideals of $S$ then $S/I$ is Cohen-Macaulay if and only if $S/J$ is Cohen-Macaulay.
\begin{Lemma}Let $I_{L}$ be the binomial edge ideal
of a line of length $n$ then $S/I_{L}:g$ is
Cohen-Macaulay of dimension $n+1.$
\end{Lemma}
\begin{proof}$I_{L}$ is a complete intersection and using $(b),(c)$ and $(d)$ of Lemma
3.2 we have $I_{L}:g$ and $J_{\tilde{G}}$ are linked ideals. Now it follows
from above theorem  that $S/I_{L}:g$ is Cohen-Macaulay because $S/J_{\tilde{G}}$ is Cohen-Macaulay.
\end{proof}

\begin{Theorem}Any cycle of length $n\geq 3$ is approximately Cohen-Macaulay.
\end{Theorem}

\begin{proof} First we will compute the depth of  $S/I_{C}$ for $n\geq 3$.
From above notations $I_{C}=(I_{L},g).$ Consider the exact sequence

$$0\rightarrow S/I_{L}:g(-2) \rightarrow S/I_{L}\rightarrow S/I_{C}\rightarrow 0.$$

Now it follows from the Depth's Lemma that
$$depth(S/I_{C})\geq \min \{depth(S/I_{L}:g)-1,depth(S/I_{L})\}=n.$$
 Hence $depth(S/I_{C})\geq n.$

 Now $S/U_{S}(I_{C})\cong S/J_{\tilde{G}}$, which is $n+1-$dimensional Cohen-Macaulay ring, so from Theorem 1.6.  $S/I_{C}$ is approximately Cohen-Macaulay.
\end{proof}

Furthermore we will find the Hilbert series of $S/I_{C}$. For this we have to introduce a monomial ideal $M=(x_{2}x_{3}\cdots x_{n-1},x_{2}x_{3}\cdots x_{n-2}y_{n-1},\dots,x_{2}y_{3}\cdots y_{n-1},y_{2}y_{3}\cdots y_{n-1}).$ We need also the expression for the Hilbert series of $S/J_{\tilde{G}}$. Recall the result in \cite{burns}, if $R$ is Cohen-Macaulay ring with dimension $d$ \ and \ $\underline{x}=x_{1},\dots,x_{d}$ be the homogenous
system of parameters of degree 1 then
$$H(R,t)=\frac{H(R/\underline{x}R,t)}{(1-t)^{d}}.$$
Now in our case $S/J_{\tilde{G}}$ is Cohen-Macaulay ring with dimension $n+1$,
$\underline{x}= x_{1},y_{1}-x_{2},\dots,y_{n-1}-x_{n},y_{n}$ is system of
parameter of degree 1 of $S/J_{\tilde{G}}$ and $$S/(\underline{x}S,J_{\tilde{G%
}})\cong K[y_{1},\dots,y_{n-1}]/(y_{1},\dots,y_{n-1})^{2}$$ therefore $H(S/(%
\underline{x}S,J_{\tilde{G}}),t)=1+(n-1)t$ and we have
$$H(S/J_{\tilde{G}},t)=\frac{1+(n-1)t}{(1-t)^{n+1}}.$$

\begin{Lemma}With the notations above we have
$$\omega (S/J_{\tilde{G}})\cong (J_{\tilde{G}},M)/J_{%
\tilde{G}}.$$
\end{Lemma}

\begin{proof}
$J_{\tilde{G}}$ is the ideal of all 2-minors of a generic $2\times n-$%
matrix which implies all 2-minors of a generic $2\times n-$matrix are zero
in $S/J_{\tilde{G}}$ hence both rows of this matrix are linearly dependent,
therefore $S/J_{\tilde{G}}\cong K[x_{1},\dots,x_{n},x_{1}t,\dots,x_{n}t].$ It is
known that \cite{burns}
$$\omega (S/J_{\tilde{G}})\cong (x_{1},y_{1})^{n-2}S/J_{\tilde{G}}$$

therefore
$$\omega (S/J_{\tilde{G}})\cong
(x_{1},x_{1}t)^{n-2}K[x_{1},\dots,x_{n},x_{1}t,\dots,x_{n}t].$$

Next we consider the monomial ideal $M$ in $S/J_{\tilde{G}}$
$$MS/J_{\tilde{G}}\cong
(x_{2}x_{3}\cdots x_{n-1},x_{2}x_{3}\cdots x_{n-1}t,\dots,x_{2}x_{3}\cdots x_{n-1}t^{n-2})K[x_{1},\dots,x_{n},x_{1}t,\dots,x_{n}t].$$

Now after multiplying $\omega (S/J_{\tilde{G}})$ by $x_{2}x_{3}\cdots x_{n-1}$
and $MS/J_{\tilde{G}}$ by $x_{1}$ respectively$,$ we see that both are
isomorphic that is
$$\omega (S/J_{\tilde{G}})\cong MS/J_{\tilde{G}}\cong (J_{\tilde{G}},M)/J_{%
\tilde{G}}.$$
\end{proof}

\begin{Theorem}$\ S/(J_{\tilde{G}},M)$ is Gorenstein of dimension $n$.
\end{Theorem}

\begin{proof}From Lemma 1.3. and 3.6. We have $$\ (S/J_{\tilde{G}})/MS/J_{\tilde{G}}
\cong S/(J_{\tilde{G}},M).$$ which is Gorenstein of dimension $n$.
\end{proof}

\begin{Lemma}With the notations above we have
$$I_{L}:g=(I_{L},M)\text{ \ } and \text{ \ }hence\text{ \ } (I_{L}:g,J_{\tilde{G}})=(M,J_{\tilde{G}}).$$
\end{Lemma}

\begin{proof}
First we will show that $M\subseteq I_{L}:g$. Because $I_{L}\subseteq I_{L}:g$ it will be enough to prove that $$(I_{L},M)/I_{L}\subseteq I_{L}:g/I_{L}$$ in $S/I_{L}$. That is, we have always $$x_{i}y_{i+1}\equiv x_{i+1}y_{i} \text{ mod }I_{L}$$ for $i=1,\dots,n-1$. Now let $x_{2}x_{3}\cdots x_{n-1}\in M$. With $g=x_{1}y_{n}-x_{n}y_{1}$, we get
$$x_{2}x_{3}\cdots x_{n-1}g\equiv x_{1}\cdots x_{n-1}y_{n}-x_{2}\cdots x_{n-1}x_{n}y_{1} \text{ mod }I_{L}.$$
Now put $x_{n-1}y_{n}\equiv x_{n}y_{n-1}\text{ mod }I_{L}$ and $x_{n-2}y_{n-1}\equiv x_{n-1}y_{n-2}\text{ mod }I_{L}$ and so on. After $(n-1)$ steps it follows that $$x_{2}x_{3}\cdots x_{n-1}g\equiv 0\text{ mod }I_{L}.$$ This proves that $x_{2}x_{3}\cdots x_{n-1}\in I_{L}:g.$ Similarly $\ y_{2}y_{3}\cdots y_{n-1}\in I_{L}:g.$
Now take any arbitrary element $\ x_{2}\cdots x_{i}y_{i+1}\cdots y_{n-1}\in M$ \ then
$$x_{2}\cdots x_{i}y_{i+1}\cdots y_{n-1}g\equiv y_{1}x_{2}\cdots x_{i}y_{i+1}\cdots y_{n-1}x_{n}-x_{1}\cdots x_{i}y_{i+1}\cdots y_{n}\text{ mod }I_{L}.$$
We have  $x_{i}y_{i+1}\equiv x_{i+1}y_{i}\text{ mod }I_{L}.$ Therefore
$$x_{2}\cdots x_{i}y_{i+1}\cdots y_{n-1}g\equiv y_{1}x_{2}\cdots x_{i}y_{i+1}\cdots y_{n-1}x_{n}-x_{1}\cdots x_{i-2}(x_{i-1}y_{i})(x_{i+1}y_{i+2})y_{i+3}\cdots y_{n}\text{ mod }I_{L}.$$
Now replace $x_{i-1}y_{i}=x_{i}y_{i-1}$ $and$ $x_{i+1}y_{i+2}=x_{i+2}y_{i+1}$ in $S/I_{L}$, we have
$$x_{2}\cdots x_{i}y_{i+1}\cdots y_{n-1}g\equiv y_{1}x_{2}\cdots x_{i}y_{i+1}\cdots y_{n-1}x_{n}-x_{1}\cdots (x_{i-2}y_{i-1})x_{i}y_{i+1}(x_{i+2}y_{i+3})\cdots y_{n}\text{ mod }I_{L}.$$
If we continue such replacements, we get that $$x_{2}\cdots x_{i}y_{i+1}\cdots y_{n-1}g\equiv 0 \text{ mod }I_{L}.$$

So $M\subseteq I_{L}:g$ hence
$$(M,J_{\tilde{G}})\subseteq (I_{L}:g,J_{\tilde{G}}).$$

$S/I_{L}$ is a Gorenstein ring and $S/I_{L}\twoheadrightarrow S/J_{%
\tilde{G}},$ therefore from Lemma 1.3
$$\omega (S/J_{\tilde{G}})\cong
Hom_{S/I_{L}}(S/J_{\tilde{G}},S/I_{L})\cong (I_{L}:J_{\tilde{G}})/I_{L}.$$
Using Lemma 3.2.(d) we get
$$\omega (S/J_{\tilde{G}})\cong (I_{L}:g)/I_{L}.$$

Now from Lemma 3.6 we have that $$\ \omega (S/J_{\tilde{G}})\cong (J_{\tilde{G}%
},M)/J_{\tilde{G}}.$$
So there are two expressions for $\omega (S/J_{\tilde{G}%
}).$ As the canonical module is unique up to isomorphism. We want to describe an isomorphism. To this end define a map $$\phi : \
(J_{\tilde{G}},M)/J_{\tilde{G}} \rightarrow  (I_{L}:g)/I_{L}$$ which sends
$\sum\limits_{i=1}^{n-1}r_{i}m_{i}+J_{\tilde{G}}$ to $\sum%
\limits_{i=1}^{n-1}r_{i}m_{i}+I_{L}$ where $m_{1},\dots,m_{n-1}$ are generators
of $M$. If $$\sum\limits_{i=1}^{n-1}r_{i}m_{i}-\sum\limits_{i=1}^{n-1}r_{i}^{\prime
}m_{i}\in J_{\tilde{G}}.$$ Then this implies $$\sum%
\limits_{i=1}^{n-1}(r_{i}-r_{i}^{\prime
})m_{i}\in J_{\tilde{G}} \cap M\subseteq J_{\tilde{G}} \cap
(I_{L},M)=I_{L}.$$ This follows because of  the inclusion $$I_{L}\subseteq J_{\tilde{G}}\cap
(I_{L},M)\subseteq J_{\tilde{G}}\cap (I_{L}:g)=I_{L}$$
as follows by Lemma 3.2 (b).
Hence it's a well define map and clearly a homomorphism.$$\phi \in
Hom_{S/J_{\tilde{G}}}(\omega (S/J_{\tilde{G}}),\omega (S/J_{\tilde{G}%
}))\cong S/J_{\tilde{G}}\text{  (see Lemma 1.3.)}.$$
 So any homomorphism $\omega (S/J_{\tilde{G}})\to \omega (S/J_{\tilde{G}})$ is given by multiplication by an element of $S/J_{\tilde{G}}$. Because $\phi $ is a non-zero homomorphism of degree zero it is in fact an isomorphism. That is $$
Im(\phi )=(I_{L},M)/I_{L}.$$  So finally we get%
$$I_{L}:g=(I_{L},M)\text{ \ } and \text{ \ }hence\text{ \ } (I_{L}:g,J_{\tilde{G}})=(M,J_{\tilde{G}})$$
\end{proof}

\begin{Lemma} The Hilbert series of $%
S/(I_{L}:g,J_{\tilde{G}})$ is
$$H(S/(I_{L}:g,J_{\tilde{G}}),t)=\frac{1+(n-1)t-(n-1)t^{n-2}-t^{n-1}}{%
(1-t)^{n+1}}$$
\end{Lemma}

\begin{proof}

It is known from \cite{burns} that $$H(\omega (S/J_{\tilde{G}%
}),t)=(-1)^{n+1}H(S/J_{\tilde{G}},\frac{1}{t}).$$
 Therefore,
$$H(\omega (S/J_{\tilde{G}}),t)=\frac{(n-1)t^{n}+t^{n+1}}{(1-t)^{n+1}}.$$

 As we know $$\omega(S/J_{\tilde{G}})\cong (M,J_{\tilde{G}})/J_{\tilde{G}}.$$ From the above formula the initial degree of $\omega(S/J_{\tilde{G}})$ is $n$ while the initial degree of $(M,J_{\tilde{G}})/J_{\tilde{G}}$ is $%
n-2.$ Therefore by dividing the non-zero divisor $t^{2}$ we get the Hilbert
series of $(M,J_{\tilde{G}})/J_{\tilde{G}}$
$$H((M,J_{\tilde{G}})/J_{\tilde{G}},t)=\frac{(n-1)t^{n-2}+t^{n-1}}{(1-t)^{n+1}}.$$

Consider the exact sequence
$$0\rightarrow (M,J_{\tilde{G}})/J_{\tilde{G}}\rightarrow S/J_{\tilde{G%
}}\rightarrow S/(M,J_{\tilde{G}})\rightarrow 0.$$
From Lemma 3.8.$$(I_{L}:g,J_{\tilde{G}})=(M,J_{\tilde{G}}).$$ Therefore $$%
H(S/(I_{L}:g,J_{\tilde{G}}),t)=H(S/J_{\tilde{G}},t)-H((M,J_{\tilde{G}})/J_{%
\tilde{G}},t).$$ Hence we get the required result.
\end{proof}

\begin{Theorem}Hilbert series of $S/I_{C}$ is
$$H(S/I_{C},t)=\frac{(1+t)^{n-1}-t^{2}(1+t)^{n-1}+(n-1)t^{n}+t^{n+1}}{%
(1-t)^{n+1}}.$$
In particular, the multiplicity of $S/I_{C}$ is $e(S/I_{C})=n.$
\end{Theorem}

\begin{proof}Hilbert series of binomial edge ideal of $I_{L}$ is easy to compute because $I_{L}$ is a complete intersection generated by $n-1$ forms of degree $2$. Namely we have
$$H(S/I_{L},t)=\frac{(1-t^{2})^{n-1}}{(1-t)^{2n}}=\frac{%
(1+t)^{n-1}}{(1-t)^{n+1}}.$$
Consider the exact sequence%
$$0\rightarrow S/I_{L}\rightarrow S/I_{L}:g\oplus S/J_{\tilde{G}%
}\rightarrow S/(I_{L}:g,J_{\tilde{G}})\rightarrow 0.$$
Therefore
$$H(S/I_{L}:g,t)=H(S/(I_{L}:g,J_{\tilde{G}}),t)+H(S/I_{L},t)-H(S/J_{\tilde{G}%
},t).$$  So by using Lemma 3.9. We get
$$H(S/I_{L}:g,t)=\frac{(1+t)^{n-1}-(n-1)t^{n-2}-t^{n-1}}{(1-t)^{n+1}}.$$

Consider the another exact sequence and replace $(I_{L},g)=I_{C}$%
$$0\rightarrow S/I_{L}:g(-2)\rightarrow S/I_{L}\rightarrow
S/I_{C}\rightarrow 0.$$

We have $$H(S/I_{C},t)=H(S/I_{L},t)-t^{2}H(S/I_{L}:g,t).$$ and after putting values we get the desired formula.
\end{proof}

\end{document}